\documentclass[12pt]{amsart}
\usepackage[utf8]{inputenc}
\usepackage{amsfonts}
\usepackage{amsmath}
\usepackage{amssymb}
\usepackage{amsthm}
\usepackage{xcolor}
\usepackage[margin=1.3in]{geometry}
\usepackage{graphicx}
\usepackage[11pt]{moresize}
\usepackage{tikz}
\usepackage[math]{cellspace}
\tikzstyle{vertex}=[circle,draw=black,fill=black,inner sep=0,minimum size=3pt,text=white,font=\footnotesize]
\usepackage{nccmath}

\newtheorem{theorem}{Theorem}
\newtheorem*{conjecture*}{Conjecture}

\newtheorem{lemma}{Lemma}

\theoremstyle{remark}
\newtheorem{definition}{Definition}[section]

\newtheorem*{remark*}{Remark}

\newcommand{\vs}{\vspace{3mm}}

\newcommand{\hs}{\hspace{1mm}}

\newcommand{\mc}{\mathcal}

\newcommand{\ep}{\epsilon}
\newcommand{\lam}{\lambda}
\newcommand{\sub}{\subseteq}
\newcommand{\al}{\alpha}

\newcommand{\wt}{\widetilde}

\newcommand{\subtext}{\underset}

\DeclareMathOperator{\Maj}{Maj}

\title{Approximate trace reconstruction of random strings from a constant number of traces}
\author{Zachary Chase \and Yuval Peres}
\thanks{The first author is partially supported by Ben Green's Simons Investigator Grant 376201 and gratefully acknowledges the support of the Simons Foundation. Research of Y.P. was partially supported by NSF grant DMS-1900008.}

\newcommand{\Addresses}{{
  \bigskip
  \footnotesize

  \textsc{Zachary Chase, University of Oxford}
  
  \par\nopagebreak
  \textit{Email}: \texttt{zachary.chase@maths.ox.ac.uk}

  \medskip
  
  \vspace{1.5mm}

  Yuval Peres, \textsc{Kent State University}\par\nopagebreak
  \textit{Email}: \texttt{yuval@yuvalperes.com}
}}

\date{July 8, 2021}

\begin{document}

\begin{abstract}
In the trace reconstruction problem, the goal is to reconstruct an unknown string $x$ of length $n$ from multiple traces obtained by passing $x$ through the deletion channel. In the relaxed problem of {\em approximate} trace reconstruction, the goal is to reconstruct an approximation $\widehat{x}$ of $x$ which is close (within $\ep n$) to $x$ in edit distance. We show that for most strings $x$, this is possible with high probability using only a constant number of traces. Crucially, this constant does not grow with $n$, and only depends on the deletion probability and $\ep$. 
\end{abstract}

\maketitle

\section{Introduction}

Given a string $x \in \{0,1\}^n$, a \textit{trace} of $x$ is obtained by deleting each bit of $x$ with probability $q$, independently, and concatenating the remaining string. For example, a trace of $11001$ could be $101$, obtained by deleting the second and third bits. The goal of the trace reconstruction problem is to determine an unknown string $x$, with high probability, by looking at as few independently generated traces of $x$ as possible. 

\vs

The problem just described, known as worst-case trace reconstruction for the deletion channel, has been extensively studied, with the current best upper bound being $\exp(\wt{O}(n^{1/5}))$ \cite{chasetrupper} and current best lower bound being $\wt{\Omega}(n^{3/2})$ \cite{chasetrlower}. 

\vs

Average-case trace reconstruction, when one wants to reconstruct a uniformly randomly chosen string from few traces, has also been analyzed. The current best upper bound is $\exp(O(\sqrt[3]{\log n}))$ \cite{hppz}, and the current best lower bound is $\wt{\Omega}(\log^{2.5}n)$ \cite{chasetrlower}.

\vs

A key motivation for the trace reconstruction problem is its application to DNA  reconstruction.   In practice, often the number of traces  available does not allow exact reconstruction, and it is natural to study what information can be inferred about the unknown string. Mitzenmacher \cite{mitzenmacher} proposed the problem of ``approximate trace reconstruction", of trying to reconstruct $x$ but allowing some error. Davies, R\'{a}cz, Rashtchian, and Schiffer \cite{drrs}, considered this problem, with the notion of ``error" being ``edit distance". So, given few traces, one wishes to output a string that is close in edit distance to the unknown one. 

\begin{definition}
The \textit{edit distance} between two strings $x,y \in \{0,1\}^{< \infty}$, denoted $d_e(x,y)$, is the minimum number of deletions and insertions needed to turn $x$ into $y$. Alternatively, if $w$ is a longest common subsequence of $x$ and $y$, then $d_e(x,y) := |x|+|y|-2|w|$ (where $|\cdot|$ denotes length). Note that $d_e$ is sometimes called the ``insertion-deletion distance". A variant that also allows for substitutions is often used; this differs from $d_e$ by at most a factor of 2 and thus is equivalent for our purposes.
\end{definition}

\vspace{1.5mm}

As noted above, \cite{drrs} considered approximate trace reconstruction. That paper gave upper bounds for   the number of traces required if the unknown string has long runs, and general lower bounds. 

\vs

In this paper, we study the ``average-case" approximate trace reconstruction question. 

\vs

\noindent{\bf Definition.} We say that $T$ traces suffice to \textit{$(q,\ep,\delta)$-approximately reconstruct a random $n$-bit string} if there exists a function $F: (\{0,1\}^{\le n})^T \to \{0,1\}^{<\infty}$ with the following property: If $x$ is chosen uniformly at random from $\{0,1\}^n$ and $\vec{U} := (U^{(1)},\dots,U^{(T)})$ is a sequence of random traces of $x$ (obtained from independent applications of the deletion channel with deletion probability $q$), then with probability at least $1-\delta$, we have $d_e(x,F(\vec{U})) < \epsilon n$. 

\vs
The point of our main theorem, that we now state, is that if $q, \delta$, and $\epsilon$ are constant, then $T$ can be chosen to be a constant (independent of $n$). 

\vs

\begin{theorem}\label{main}
For all $q,\delta \in (0,1)$, there is some $C = C(q,\delta) \ge 1$ so that for any $\ep > 0$ and any $n \ge 1$, it holds that $T := \exp(C\sqrt[3]{\log(1/\ep)})$ traces suffice to $(q,\ep,\delta)$-approximately reconstruct a random $n$-bit string.
\end{theorem}

\vs

We wish to note that our theorem does apply for any $\epsilon$ as a function of $n$.

\vs

Our theorem recovers the non-computational part of the main result of \cite{bls}, that there is a binary code of rate $1-\ep$ that is recoverable from $\exp(C\sqrt[3]{\log(1/\ep)})$ traces. See Section 11 for more details.

\vs

\section{Related Work}

Holenstein, Mitzenmacher, Panigrahy, and Wieder \cite{hmpw} established  that $\exp(\wt{O}(n^{1/2}))$ traces suffice for worst-case trace reconstruction of an $n$-bit string. Nazarov and Peres \cite{nazarovperes} and De, O'Donnell, and Servedio \cite{dds} simultaneously obtained that $\exp(O(n^{1/3}))$ traces suffice. This was improved to 
$\exp(\wt{O}(n^{1/5}))$ by Chase~\cite{chasetrupper}.
\vs

For the same problem, Batu, Kannan Khanna, and McGregor~\cite{bkkm} proved a lower bound of $\Omega(n)$, which was improved to $\wt{\Omega}(n^{5/4})$ by Holden and Lyons \cite{holdenlyons}, and then to $\wt{\Omega}(n^{3/2})$ in \cite{chasetrlower}.

\vs

For average-case trace reconstruction,   a polynomial upper bound   was proved   in \cite{hmpw} for sufficiently small $q$. Peres and Zhai \cite{pereszhai} obtained a sub-polynomial upper bound of $\exp(O(\log^{1/2} n))$ for $q < \frac{1}{2}$ which was then improved to $\exp(O(\log^{1/3} n))$ for all (constant) $q$ by Holden, Pemantle, Peres, and Zhai \cite{hppz}. 

\vs

We will need a slight generalization of the result \cite{hppz}. We need the result to apply to the case that the unknown string is a concatenation of a short deterministic string followed by a random string. We also need the statement that $\exp(C\log^{1/3}k)$ traces suffice to reconstruct the first $k$ bits of the unknown (pseudorandom) string.

\vs

\begin{lemma}\label{hppzgeneralization}
For all $q \in (0,1)$ and $C' \ge 1$, there is some $C \ge 1$ so that the following holds for all $1 \le k \le n$. Let $w$ be a (deterministic) string of length at most $C'\log k$. Let $x$ be a string formed by appending $n-|w|$ uniformly random bits from $\{0,1\}$ to $w$. Then, with probability at least $1-1/k$ over the choice of $x$ and the draw of $\exp(C\log^{1/3}k)$ traces, we can reconstruct the first $k$ bits of $x$. 
\end{lemma}

\vs

The generalization to allow a deterministic prefix follows straightforwardly from the proof given in \cite{hppz}, since they use worst-case methods to reconstruct the prefix and don't use the prefix afterwards. Indeed, for coarse  alignment of bit $\ell$ in $x$ for $\ell< \log^{5/3}(k)$, estimating the corresponding location in the trace to be $p\ell$ incurs an error greater than $\log k$ with probability at most $\exp(-c \log^{1/3} (k))$ by Hoeffding's inequality (see also page 36 in arXiv version 1 of \cite{hppz}). The generalization to reconstructing the first $k$ bits of an $n$-bit string follows from the proof given in \cite{hppz} by applying the test in the first alignment step only to blocks contained in the first $2k$ bits of the trace, say. 

\vs

Holden and Lyons \cite{holdenlyons} proved a lower bound for this random variant of $\wt{\Omega}(\log^{9/4}n)$, which was then improved to $\wt{\Omega}(\log^{5/2}n)$ in \cite{chasetrlower}.

\vs

Several other variants of the trace reconstruction problem have been considered. 

\vs

Coded trace reconstruction has been studied, in which one wants to find a (large) subset of strings so that trace reconstruction is easier when the unknown string is promised to be among this subset. See \cite{cgor},\cite{bls}.  

\vs

The population recovery variant of trace reconstruction has been studied, in which there is an unknown distribution over the unknown string (usually of small support) and one wants to determine this distribution up to small error given many random traces of a random string chosen according to this distribution. See \cite{bcfrs}, \cite{bcss}, \cite{narayanan}.  

\vs

A smoothed variant of trace reconstruction,  interpolating between worst-case and average-case trace reconstruction, has been considered in \cite{cdlss}. In this variant, the input $x$ to the deletion channel is a random perturbation of an unknown arbitrary string $X_*$, and we want to reconstruct $x$ from   random traces of it.  

\vs

Circular trace reconstruction, in which the unknown string undergoes a random cyclic shift before each time it is   passed through the deletion channel,  was analyzed in \cite{circulartr}.

\vs

One of the main reasons for study of trace reconstruction is its applications to biology. Trace reconstruction has found useful connections with DNA data storage \cite{lcaplus} and immunogenomics \cite{bfrps}. See \cite{bprs} for a nice survey. See \cite{katoh} for a focus on multiple alignment of DNA sequences. 

\vs

There have also been generalizations to objects other than strings. For matrices, see \cite{kmmp}. For trees, see \cite{drr}, \cite{mr}, and \cite{br}. 

\vs

It is natural to ask how effectively one can do trace reconstruction to distinguish pairs of strings with small Hamming or edit distance. For this topic,  and for the limitations of mean-based arguments, see \cite{simabruck}, \cite{cdrv}, \cite{gsz}. 

\vs

\section{Notation}

We use $C_0,C_1,\dots,K_0,K_1,\dots$ to denote large constants, and $c_0,c_1,\dots,\kappa_0,\kappa_1,\dots$ to denote small constants. They are functions of $q,\delta,\ep$ but not of $n$; they are only polynomial in $1/\ep$. We reserve $C,c$ for constants that we do not wish to emphasize, that change from line to line.

\vs

Let $x[a:b]$ denote the contiguous substring $x_a,\dots,x_{b-1}$. For an interval $I \sub [a,b]$, the length of $I$ is defined to be $b-a+1$. For a string $W$, define $\Maj(W)$ to be $1$ if there are strictly more $1$s than $0$s in $W$, and $0$ otherwise.

\vs

Let $x$ be a string. For a trace $U$ of $x$ and an index $i \in |U|$, let $g_U(i)$ be the index in $x$ that was sent to index $i$ under the deletion channel. To fix notation, a trace $U$ of $x$ is determined by the values of the ``retention variables" $(r_i^U)_{i=1}^n$, where $r_i^U$ is $1$ if bit $i$ in $x$ is kept, and $0$ otherwise. Note that $g_U(j) = l$ if and only if $\sum_{i=1}^l r_i^U = j$. 

\vs

We will mostly use $\mathbb{P}$ to denote a probability. When we want to be more specific about exactly what the randomness present is, we will use the following notation. We use $\mathbb{P}^X$ to denote the uniform distribution over $\{0,1\}^n$ (i.e., the distribution according to which $x$ is chosen). We use $\mathbb{P}^R$ to denote the Bernoulli distribution over $\{0,1\}^n$ where $0$ has probability $1-p$ and $1$ has probability $p$ (i.e., the distribution of the retention variables). 

\vs

\section{Sketch of Argument}

The proof of Theorem \ref{main} uses three ingredients. 

\vs

The first is to align traces to each other. This means that, for an index $j$ in a trace, we want to locate, in each of the other traces, an index $j'$ such that $j$ and $j'$ came from roughly the same location in the original string. We align the traces to each other by iterating a majority bit argument, capitalizing on the randomness of the unknown source string and the fact that two independently random strings won't have the same majority bit in various contiguous substrings noticeably more than half of the time.

\vs

The second ingredient is the use of \textit{anchors}, an idea first appearing in \cite{hmpw}. We will be reconstructing large chunks of the unknown string near locations where there is one large block of $0$s except for exactly one $1$, which is in the middle. The reason we look for such locations is that, if the $1$ survives the deletion channel, it is easily identifiable in the trace, and thus serves as a starting point for reconstruction. 

\vs

The third ingredient is that $\exp(C\log^{1/3}m)$ traces suffice to reconstruct the first $m$ bits of a random, unknown $n$-bit string with a small deterministic prefix. This is the content of Lemma \ref{hppzgeneralization} and follows essentially already from the work of \cite{hppz}.

\vs

\section{Test}

Let $\al \in (\frac{1}{2}-\frac{1}{100},\frac{1}{2})$ be a parameter to be chosen later. We are given two strings $U,V \in \{0,1\}^m$. We now describe whether the test declares a match between $U$ and $V$. Briefly, we divide each string into $m^\al$ blocks (each of length $m^{1-\al}$) and declare there to be a match if at least $\frac{1}{2}+\kappa_0$ proportion of the blocks have the same majority bit. 

\vs

Let $I_j = ((j-1)m^{1-\al},jm^{1-\al}]$ for $1 \le j \le m^\al$. Let $$\sigma_j = 1\Big[\Maj(U[I_j]) = \Maj(V[I_j])\Big].$$ We declare there to be a \textit{match} between $U$ and $V$ if $$\sum_{j=1}^{m^\al} \sigma_j > (\frac{1}{2}+\kappa_0)m^\al,$$ where $\kappa_0$ is a small positive constant, to be specified later.

\vs

\section{Properties of Test}

Fix parameters $\lambda \in (\frac{1}{2},1-\al)$ and $\beta \in (0,\frac{1-\al}{2})$ to be chosen later. 

\vs

We now introduce two lemmas. The first says that the probability that two blocks in traces do not match given that they come from nearly the same part of the unknown string is at most $\exp(-cm^{2\lam-1})$. The second lemma says that the probability that two blocks in traces do match given that they come from far away parts of the unknown string (measured by $\beta$) is at most $\exp(-cm^\al)$. 

\vs

\begin{lemma}\label{truematch}
Fix intervals $I,J \sub [n]$ with left endpoints denoted $i,j$, respectively. Let $x$ be a random string and $U,V$ random traces of $x$. Let $A_1$ denote the event that $|g_U(i)-g_V(j)| \le m^\lambda$. Then $$\mathbb{P}(U[I],V[J] \text{ match} | A_1) \ge 1-\exp(-cm^{2\lambda-1}).$$
\end{lemma}

\begin{proof}
Let $A$ denote the event $|g_U(i+k)-g_V(j+k)| \le 3m^\lambda$ for all $0 \le k \le m-1$. We first show \begin{equation}\label{Alarge} \mathbb{P}(A) \ge 1-6m\exp(-c_1m^{2\lambda-1}).\end{equation} We have $$\left|g_U(i+k)-g_V(i+k)\right| \le \left|g_U(i+k)-g_U(i)-\frac{k}{p}\right|+\left|-g_V(i+k)+g_V(i)+\frac{k}{p}\right|+\left|g_U(i)-g_V(i)\right|.$$ Note $g_U(i+k)-g_U(i) - \frac{k}{p} \ge m^\lambda$ implies $\sum_{j=1}^{k/p+m^\lambda} r_{i+j} \le k$, which has probability at most $ \exp(-cm^{2\lambda}/k) \le \exp(-cm^{2\lambda-1})$. Similarly if $g_U(i+k)-g_U(i)-\frac{k}{p} \le -m^\lambda$. So, by a union bound, we've established \eqref{Alarge}. 

\vs

Let $\varphi(k) = km^{1-\al}$ for $0 \le k \le m^\al$. For $1 \le k \le m^\al$, let $$\ell_k = \max\left(g_U(i+\varphi(k-1)),g_V(j+\varphi(k-1))\right).$$ Let $B$ denote the event that $\sum_{\ell=\ell_k}^{\ell_k+\frac{m^{1-\al}}{p}-3m^\lam} r_\ell^U r_\ell^V \ge \frac{1}{2}pm^{1-\al}$ for all $1 \le k \le m^\al$. We claim that \begin{equation}\label{Blarge} \mathbb{P}(B | A) \ge 1-\exp(-c'm^{1-2\al}). \end{equation} Since $\lam < 1-\al$, the inequality \eqref{Blarge} follows immediately from Höeffding's inequality and a union bound. 

\vs

Now, fix $r_i^U$'s and $r_i^V$'s such that $A$ and $B$ hold. Then, by the definition of $\ell_k$ (and that $B$ holds), the strings $U[i+I_k]$ and $V[j+I_k]$ consist of at least $\frac{pm^{1-\al}}{2}$ identical bits and the remainder are independent. Hence, $$\mathbb{P}^X\left(\sigma_k = 1 | A,B\right) \ge c(p) > 0$$ for some constant $c(p)$ depending only on $p$. Therefore, fixing $r_i^U$'s and $R_i^V$'s so that $A$ and $B$ hold, the $\mathbb{P}^X$-probability that a match is declared between $U$ and $V$ is at least $1-\exp(-cm^\al)$, which is at least $1-\exp(-cm^{2\lam-1})$. The proof of Lemma \ref{truematch} is complete.  

\end{proof}

\begin{lemma}\label{falsematch}
Fix intervals $I,J\sub [n]$ with left endpoints denoted $i,j$, respectively. Let $x$ be a random string and $U,V$ random traces of $x$. Let $B_1$ denote the event that $g_U(i) > g_V(j)+m^{1-\beta}$. Then, $$\mathbb{P}(U[I], V[J] \text{ match } | B_1) \le 2\exp(-cm^\al).$$
\end{lemma}

\begin{proof}
Let $B$ denote the event that $g_U(i+l) \ge g_V(j+l+m^{1-\al})+cm^{1-\al}$ for all $l \le m-m^{1-\al}$. Then $\mathbb{P}(B | B_1) \ge 1-2me^{-cm^{1-2\beta}}$. Given $B$, the bits in $U[i+I_k]$ are independent of $V[j+I_k]$ for $l \le k$ and $U[i+I_l]$ for $l < k$. So $\mathbb{P}(U[I] \text{ matches } V[J] \hs\hs | \hs\hs B) \le \exp(-cm^\al)$.  
\end{proof}

\vs

\section{Alignment Algorithm}

Let $K_1 = (1/\ep)^{100}$. We will define an associated $K_2$, which will satisfy $\sqrt{K_1} \le K_2 \le K_1$. The edit distance of our approximation will be at most $n/K_2^{0.1}$, which is at most $\ep n$.

\vs

Let $U^{(1)},\dots,U^{(T)}$ be the $T$ traces. Fix $j \in [\ep^2n,(\frac{1}{2}-\ep^2)n]$. The following algorithm will all be with respect to $j$ (and the $U^{(t)}$'s). 

\vs

Let $L_1 = \sqrt{n}$ and $L_{r+1} = \sqrt{L_r}$ for $r \ge 1$. For $r \ge 1$, let $I_r^{(1)} = [j-2L_r,j-L_r]$. For each $t \in \{2,\dots,T\}$, let $I_1^{(t)}$ be the leftmost interval of length $L_1$ that matches $I_1^{(1)}$, with $I_1^{(t)} := \emptyset$ if no such interval exists. If $I_1^{(t)} = \emptyset$, let $\wt{I}_1^{(t)} = \emptyset$. Otherwise, writing $I_1^{(t)} = [a,b]$, define $\wt{I}_1^{(t)} = [b,b+2L_1]$. Inductively, for $r \ge 1$ and $t \in [2,T]$, let $I_{r+1}^{(t)}$ denote the first interval of length $L_{r+1}$ in $\wt{I}_r^{(t)}$ that matches $I_{r+1}^{(1)}$ (let $I_{r+1}^{(t)} := \emptyset$ if no such interval exists). If $I_{r+1}^{(t)} = \emptyset$, let $\wt{I}_{r+1}^{(t)} = \emptyset$, and otherwise, if $I_{r+1}^{(t)} = [a,b]$, let $\wt{I}_{r+1}^{(t)} = [b,b+2L_{r+1}]$. With our chosen $K_1$, we stop at $R := \min\{r : L_r \le K_1\}$ and define $K_2 := L_R$. That is, for each $t \ge 2$, the last intervals defined are $I_R^{(t)}$ and $\wt{I}_R^{(t)}$, obtained by matching an interval of length $K_2$ in $\wt{I}_{R-1}^{(t)}$ with $I_R^{(1)}$. 

\vs

Given an interval $J$ and a trace $U$, we say the block $U[J]$ is a \textit{trace-anchor} if it has length\footnote{We needed to insist the length was odd so that the ``middle" is well-defined. But we shall merely refer to the length as $(4+\frac{1}{2}c_0p)C_9\log K_2$ in what follows. Similarly with other similar definitions.} $2\lfloor (2+\frac{1}{4}c_0p) C_9\log K_2\rfloor+1 \approx (4+\frac{1}{2}c_0p)C_9\log K_2$ and contains exactly one $1$, and this $1$ is located exactly in the middle. Here, $c_0 \in (0,1/100)$. We say a block in a trace is a \textit{trace-pseudo-anchor} if it has length $C_9 \log K_2$ and contains at most $1.5 C_8\log K_2$ ones. See the table on page $8$ for a consolidation of many of the definitions.

\vs

For a $t \in \{2,\dots,T\}$, we say that $t$ is \textit{$j$-good} if $U^{(t)}\left[\wt{I}_R^{(t)}\right]$ contains exactly one trace-anchor, say $U^{(t)}[J]$, and no trace-pseudo-anchor $U^{(t)}[J']$ with $J'$ disjoint from $J$. Furthermore, we say that $t$ is \textit{$j$-spurious} if $U^{(t)}\left[\wt{I}_R^{(t)}\right]$ contains a block of length at least $C_9\log K_2$ with at least $2$ ones and at most $\frac{1}{2}\log K_2$ ones.

\vs

We say that an index $j$ of $U^{(1)}$ is \textit{trace-useful} if $\#\{t \le T : t \text{ is } j-\text{good}\} \ge \frac{1}{2}pT$ and $\#\{t \le T : t \text{ is } j-\text{spurious}\} \le \frac{1}{8}p^2T$. 

\vs

If $j$ is trace-useful, then for each $j$-good number $t$, delete all bits in $U^{(t)}$ before and including the middle of the trace-anchor to obtain a substring $U_*^{(t)}$ (that starts with at least $\frac{1}{2}(4+\frac{1}{2}c_0p)C_9\log K_2$ $0$'s and was preceded by a $1$). Then apply Lemma \ref{hppzgeneralization} to the traces $U_*^{(t)}$ for $t$ which is $j$-good.

\vs

\section{Analysis of Alignment Algorithm}

We start by arguing that with extremely high probability, a match was found at each step of the iteration and that each of the blocks $U^{(t)}[\wt{I}_R^{(t)}]$ come from basically the same place as $U^{(1)}[j-L_R,j+L_R]$. 

\begin{lemma}\label{success}
With probability at least $1-\exp(-cK_2^{2\lam-1})$, the following two both hold. For all $2 \le t \le T$, we have $\wt{I}_R^{(t)} \not = \emptyset$. For all $2 \le t \le T$ and all $0 \le k \le 2L_R-1$, we have $$\left|g_{U^{(t)}}[\ell_t+k]-g_{U^{(1)}}[\ell_1+k]\right| \le K_2^{1-\beta/2},$$ where $\ell_t$ is the left endpoint of $\wt{I}_R^{(t)}$.
\end{lemma}

\begin{proof}
For just this proof, for $t \ge 2$ and $r \ge 1$, we say that $I_r^{(t)}$ is \textit{on track} if $I_r^{(t)} \not = \emptyset$ and $|g_{U^{(t)}}(\ell_{r,t})-g_{U^{(1)}}(\ell_{r,1})| \le L_r^{1-\beta}$, where $\ell_{r,t}$ is the left endpoint of $I_r^{(t)}$. For each $r,t \ge 2$, we have, by Lemmas \ref{truematch} and \ref{falsematch}, that $$\mathbb{P}(I_r^{(t)} \text{ is not on track} \hs | \hs I_{r-1}^{(t)} \text{ is on track}) \le \exp(-cL_r^{2\lam-1})+2L_{r-1}\exp(-cL_r^\al).$$ By our choice of the $L_r$'s and the fact that $\al > 2\lam-1$, we have $$\exp(-cL_r^{2\lam-1})+2L_{r-1}\exp(-cL_r^\al) \le \exp(-c'L_r^{2\lam-1}).$$ Thus, $$\mathbb{P}(\exists t \le T : I_R^{(t)} \text{ is not on track}) \le T\exp(-c'K_2^{2\lam-1}) \le \exp(-c''K_2^{2\lam-1}).$$
\end{proof}

\vs

Say a block in the unknown string $x$ is an \textit{anchor} if it is of length $4\lfloor p^{-1}C_9\log K_2\rfloor +1$ and contains exactly one $1$, and this $1$ is located exactly in the middle. Say a block in $x$ is a \textit{pseudo-anchor} if it is of length $2p^{-1}C_9\log K_2$ and has at most $p^{-1}\log K_2$ $1$'s. Here, $C_9 \gg C_8 \gg 1$. 

\vs

Say an index $\ell$ (in the unknown string) is \textit{useful} if in $x[\ell-K_2,\ell+K_2]$ there is exactly one anchor $x[I]$ and no pseudo-anchor $x[I']$ with $I'$ disjoint from $I$. 

\vs

Since there is a lot of terminology, we put the key terms into a table, shown below.

\begin{table}[h]
\begin{tabular}{|l|l|l|}
                    & Length                           & Number of 1s               \\
Trace anchor        & $(4+\frac{1}{2}c_0p)C_9\log K_2$ & $= 1$                      \\
Trace pseudo-anchor & $C_9\log K_2$                   & $\le 1.5C_8\log K_2$             \\
Anchor              & $4p^{-1}C_9\log K_2$             & $= 1$                      \\
Pseudo-anchor       & $2p^{-1}C_9\log K_2$              & $\le p^{-1}\log K_2$
\end{tabular}
\end{table}

\vs

The parameters are chosen so that (part of) a trace anchor is very likely to have come from an anchor, and (part of) a pseudo-anchor is very likely to have come from a trace pseudo-anchor. 

\vs

\begin{table}[h]
\begin{tabular}{l|l}
$t$ $j$-good     & $U^{(t)}[\wt{I}_R^{(t)}]$ has 1 trace anchor and no disjoint trace pseudo-anchor                     \\
$t$ $j$-spurious & $U^{(t)}[\wt{I}_R^{(t)}]$ has block of length $C_9\log K_2$ with $\#1's \in [2,\frac{1}{2}\log K_2]$ \\
$j$ trace-useful & $\#\{t \text{ : } j-\text{good}\} \ge pT/2$ and $\#\{t \text{ : } j-\text{spurious}\} \le p^2T/8$.       \\
$l$ useful       & $x[\ell-K_2,\ell+K_2]$ has 1 anchor and no disjoint pseudo-anchor                                   
\end{tabular}
\end{table}

\vs

By the above paragraph, the parameters are chosen so that a trace-useful index is very likely to have come from a useful index. We will be using trace-useful indices to reconstruct part of the unknown string. We can ensure trace-useful indices are somewhat common by looking at the image of a stronger version of a useful index under the deletion channel; this is the content of Lemma \ref{useful&trace-useful}. 

\vs

Write $g = g_{U^{(1)}}$ (so, to recall, $g(j')$ is the location in $x$ which goes to bit $j'$ in $U^{(1)}$ under the deletion channel). Call an index $g(j)$ \textit{super-useful} if there is an anchor of length $(4p^{-1}+c_0)C_9\log K_2$ centered at $g(j)$ with no block in $[g(j)-K_2,g(j)+K_2]$ that has length $2p^{-1}C_9\log K_2$, has at most $2p^{-1}C_8\log K_2$ ones, and is disjoint from the anchor centered at $g(j)$. The point is that, $g(j)$ being super-useful is enough to imply, with high probability, that $j$ is trace-useful (while $g(j)$ being useful is not). 

\vs

\begin{lemma}\label{useful&trace-useful}
It holds that $$\mathbb{P}[g(j) \text{ is useful and } j \text{ is trace-useful}] \ge \frac{1}{4}2^{-(4p^{-1}+c_0)C_9\log K_2}.$$
\end{lemma}

\begin{proof}
Note that $$\mathbb{P}[\text{an anchor of length } (4p^{-1}+c_0)C_9\log K_2 \text{ is centered at } g(j)] = 2^{-(4p^{-1}+c_0)C_9\log K_2}.$$ Since $C_9 \gg C_8$, with probability at most $1/20$, the block $$[g(j)-K_2,g(j)+K_2]\setminus[g(j)-2p^{-1}C_9\log K_2, g(j)+2p^{-1}C_9\log K_2]$$ contains a block of length $2p^{-1}C_9\log K_2$ with at most $2p^{-1}\log K_2$ ones. Therefore, $$\mathbb{P}[g(j) \text{ is super-useful}] \ge \frac{1}{2}2^{-(4p^{-1}+c_0)C_9\log K_2}.$$ Note that all of the alignment was done to the left of $g(j)-K_2$, so having $|g_{U^{(t)}}[\ell_t+k]-g_{U^{(1)}}[\ell_1+k]| \le K_2^{1-\beta/2}$ for each $0 \le k \le 2K_2-1$, where $\ell_t$ is the left endpoint of $\wt{I}_R^{(t)}$, is independent of whether there is a super-useful index at $g(j)$, so by Lemma \ref{success}, the probability of both is at least $\frac{1}{3}2^{-(4p^{-1}+c_0)C_9\log K_2}$. 

\vs

Given that $g(j)$ is super-useful and the anchored $1$ survived the deletion channel in the trace $U^{(t)}$, we claim that, with conditional probability at least $3/4$, the string $U^{(t)}[I_R^{(t)}]$ contains exactly one trace-anchor with no trace-pseudo-anchor disjoint from it. To see this, first note that there was an anchor of length $(4p^{-1}+c_0)C_9\log K_2$ centered at $g(j)$, and the length of a trace-anchor is $(4+\frac{1}{2}c_0p)C_9\log K_2$. Also, note that if there were a trace-pseudo-anchor disjoint from the trace-anchor, then there is a block of length $C_9\log K_2$ with at most $1.5C_8\log K_2$ ones in $U^{(t)}[I_R^{(t)}]$, which has a very high probability of coming from a block of length at least $2p^{-1}C_9\log K_2$ with at most $2p^{-1}1.5 C_8\log K_2$ ones, but this doesn't exist in $x$ restricted to $g_{U^{(t)}}[I_R^{(t)}]$.

\vs

Since $C_9 \gg 1$, it is also very unlikely that $[g(j)-K_2,g(j)+K_2]$ contains block of length at least $\frac{1}{2}p^{-1}C_9\log K_2$ with at least $2$ ones and at most $p^{-1}\log K_2$ ones, so we may condition on this event as well. Similar reasoning to the previous paragraph shows that, with probability at least $1-\exp(-cT)$, there will then be fewer than $p^2T/8$ spurious numbers $t$.
\end{proof}

\vs

Next we show that given $j$ is trace-useful, with high probability $g(j)$ is useful.

\vspace{1.5mm}

\begin{lemma}\label{notuseful&trace-useful}
We have $$\mathbb{P}[g(j) \text{ not useful and } j \text{ is trace useful}] \le \exp(-cT).$$
\end{lemma}

\begin{proof}
We start by splitting into cases based on the definitions: $$\mathbb{P}[g(j) \text{ not useful and } j \text{ is trace useful}] \le $$ $$\mathbb{P}[\text{no anchor in } x[g(j)-K_2,g(j)+K_2] \text{ and } j \text{ trace-useful}] +$$ $$\mathbb{P}[\exists \; \text{pseudo-anchor in } x[g(j)-K_2,g(j)+K_2] \text{ and } j \text{ trace-useful}].$$ We may easily bound the second term by $$\mathbb{P}[\exists \; \text{pseudo-anchor in } x[g(j)-K_2,g(j)+K_2] \text{ and } j \text{ trace-useful}] \le \exp(-cT).$$ And since $${l \choose 2} p^2 q^{l-1} \ge \frac{p^2}{6}$$ for all $l \ge 2$ and $p+q = 1$, we may bound the first term by $$\mathbb{P}[\text{no anchor in } x[g(j)-K_2,g(j)+K_2] \text{ and } j \text{ trace-useful}] \le \exp(-cT).$$ 
\end{proof}

\begin{remark*}
Since we are choosing $T = \exp(C\sqrt[3]{\log K_2})$, the bound $\exp(-cT)$ is less than the lower bound of $K_2^{-(4+pc_0)p^{-1}C_9}$ we obtained for $\mathbb{P}[g(j) \text{ useful and } j \text{ trace-useful}]$ (in Lemma \ref{useful&trace-useful}). 
\end{remark*}

\vspace{1.5mm}

\begin{definition}
Fix $t \le T$ that is $j$-good. If $g(j)$ is useful and $j$ is trace-useful, let $\Gamma,\gamma_t$ denote the indices of the $1$ in the anchor and trace-anchor, respectively. 
\end{definition}

\begin{lemma}\label{correspond}
For any fixed $t \in [2,T]$, it holds that $$\mathbb{P}\left[g_{U^{(t)}}(\gamma_t) \not = \Gamma \hs \hs | \hs\hs  g(j) \text{ useful, } j \text{ trace-useful, and } t \text{ $j$-good} \right] \le \frac{1}{K_2^{0.5}}.$$  
\end{lemma}

\begin{proof}
We split the proof of this claim into two cases. For ease, let $A$ denote the event that $g(j)$ is useful and $j$ is trace-useful. We more directly show $$\mathbb{P}\left[A\cap \{t \text{ is } j-\text{good}\} \cap \{g_{U^{(t)}}(\gamma_t) \not = \Gamma\}\right] \le \frac{1}{K_2^{0.5}}\mathbb{P}\left[A\cap\{t \text{ is } j-\text{good}\}\right].$$ If $t$ is a $j$-good trace, then there is a trace-anchor with interval $$[a_t,b_t] := \left[\gamma_t-(2+\frac{1}{2}c_0p)C_9\log K_2,\gamma_t+(2+\frac{1}{2}c_0p)C_9\log K_2\right].$$ 

\vs

The first case is that $g_{U^{(t)}}(b_t)-g_{U^{(t)}}(a_t) \ge 2p^{-1}C_9\log K_2$. Intuitively, this case corresponds to the case that we identified the remnants of a near pseudo-anchor that was not a pseudo-anchor due to having too many ones, and nearly all of the ones were deleted when a trace was taken. We start though with the general $$\mathbb{P}[g_{U^{(t)}}(\gamma_t) \not = \Gamma \hs | A]$$ $$ \le \sum_{z \in [g(j)-K_2,g(j)+K_2]\setminus \Gamma} \mathbb{P}[g_{U^{(t)}}(\gamma_t) = z \hs | \hs A] + \frac{e^{-K_2^{1-\beta}}}{\mathbb{P}(A)},$$ the second term coming from Lemma \ref{falsematch}. Since $\mathbb{P}(A) \ge \text{poly}(1/K_2)$, it suffices to show that each summand is $O(K_2^{-1.5})$ due to the situation of the first case. Intuitively, we expect an $O(K_2^{-1.5})$ bound since deleting nearly all of at least $1.5p^{-1}\log K_2$ ones has probability $O(K_2^{-1.5})$. But we must be careful, as we are conditioning on $A$, which might affect some things. So we define a matching within $A$. Specifically, we map each such configuration to the set of configurations where the achored $1$ is retained and the retention variables of the ones in the interval that's mapped to the trace anchored 1 are arbitrary. 

\vs

The second case is that $g_{U^{(t)}}(b_t)-g_{U^{(t)}}(a_t) < 2p^{-1}C_9\log K_2$. Intuitively, this case corresponds to the case that we identified the remnants of a near pseudo-anchor that was not a pseudo-anchor due to having too small a length, and the length didn't shrink (even close to) as expected when a trace was taken. In this case, $$\mathbb{P}\left[A\cap\{t \text{ is } j-\text{good}\}\cap\{g_{U^{(t)}}(b_t)-g_{U^{(t)}}(a_t) < 2p^{-1}C_9\log K_2\}\right]$$ $$ \le \mathbb{P}\left[x[j-K_2,j+K_2] \text{ has anchor}\cap\{\exists \nu \in [j-K_2,j+K_2] : \sum_{i=1}^{2p^{-1}C_0\log K_2} r_{\nu+i} > 4C_9\log K_2\}\right],$$ which by independence is $$\le 2K_22^{-4p^{-1}C_9\log K_2}2K_2(\frac{e}{4})^{2C_9\log K_2},$$ where we used the general bound $\mathbb{P}[X > 2\mathbb{E}(X)] \le (e/4)^{\mathbb{E}(X)}$ (see, for example, Theorem 8.1.12 of Alon-Spencer \cite{alonspencer}). And this is at most $$\frac{1}{K_2}\mathbb{P}[A]$$ by the lower bound on $\mathbb{P}(A)$ in Lemma \ref{useful&trace-useful} (since $c_0 < 1/100$).
\end{proof}

\vs

\section{Reconstruction Algorithm}

We may assume that $n \ge (1/\ep)^{C_{19}}$ (where $C_{19}$ is large), for otherwise, we could simply fully reconstruct the random $n$-bit string. With this assumption, we receive $\exp(C_{20}\sqrt[3]{\log(1/\ep)})$ traces and run the alignment algorithm ($C_{20}$ is chosen last).

\vs

Let $C_{10}$ be large. We begin by finding the first index $j_1$ in $U^{(1)}$ that is trace-useful and apply the algorithm to index $j_1$ to obtain a string $\widehat{x}_{j_1}$ of length $K_2^{C_{10}}$. Inductively, with $j_r$ chosen, we find the first index $j_{r+1} > j_r+pK_2^{C_{10}}+pK_2^{3C_{10}/4}$ in $U^{(1)}$ that is trace-useful and apply the algorithm to index $j_{r+1}$ to obtain a string $\widehat{x}_{j_{r+1}}$ of length $K_2^{C_{10}}$.

\vs

Say $j_1,\dots,j_m$ were the trace-useful indices that we reconstructed based on. We then output the concatenation $\widehat{x} := \widehat{x}_{j_1}\dots \widehat{x}_{j_m}$ (as our approximation to $x$).

\vs

\section{Analysis of Reconstruction Algorithm}

Let $\widehat{x}_{j_1},\dots,\widehat{x}_{j_m}$ be the outputted strings that we concatenate to get $\widehat{x}$. 

\vs

We present a sequence of insertions and deletions to transform $\widehat{x}$ into $x$. We show, with high probability, the number of insertions and deletions we make is at most $n/K_2^{0.1}$.

\vs

We first delete all $\widehat{x}_{j_i}$'s where $g_{U^{(1)}}(j_i) \le g_{U^{(1)}}(j_{i-1})+K_2^{C_{10}}+K_2^{C_{10}/10}$. Note with high probability, we delete at most $e^{-K_2^{C_{10}/100}}n$ of the $\widehat{x}_{j_i}$'s, since the expected number of pairs $(j,j')$ with $j' > j+pK_2^{C_{10}}+pK_2^{3C_{10}/4}$ and $g_{U^{(1)}}(j') \le g_{U^{(1)}}(j)+K_2^{C_{10}}+K_2^{C_{10}/10}$ is at most $e^{-K_2^{C_{10}/100}}n$. Therefore, with high probability, we delete at most $2K_2^{C_{10}}e^{-K_2^{C_{10}/100}}n \le \frac{1}{3}\frac{1}{K_2^{0.1}}n$ bits.   

\vs

Let $\widehat{x}_{r_1},\dots,\widehat{x}_{r_s}$ be the remaining $\widehat{x}_{j_i}$'s. Delete each of the $\widehat{x}_{r_i}$'s that were falsely reconstructed (in that the application of Lemma \ref{hppzgeneralization} was not successful). We then just need to show, for example, that at most a proportion of $\frac{1}{3}\frac{1}{K_2^{0.1}}$ of the $\widehat{x}_{r_i}$'s were falsely reconstructed. Of course, it suffices to work with the original $\widehat{x}_{j_i}$'s. By Markov, it suffices to show that the probability of a fixed $\widehat{x}_{j_i}$ having been falsely reconstructed is at most $\frac{1}{3K_2^{0.2}}$. And this follows from Lemmas \ref{notuseful&trace-useful},\ref{success}, and \ref{hppzgeneralization}. 

\vs

Now we have a subsequence of $x$ and with high probability deleted at most $\frac{1}{K_2^{0.1}}n$ bits of $x$. So it just suffices to show, that, with high probability, there are at most $\frac{1}{K_2^{0.1}}n$ bits missing from this subsequence. First note that, since $K_2^{3C_{10}/4} \ll K_2^{C_{10}}$, with high probability, there are at most $n/\text{poly}(K_2)$ bits missing due to the required waiting $pK_2^{C_{10}}+pK_2^{3C_{10}/4}$ indices before looking for the next trace-useful index. We now must control how long we wait after those required waiting times. By Lemma \ref{useful&trace-useful} and Markov, with high probability, it holds that the total weighting time until seeing the first trace-useful index after the required waiting time is at least $n/K_2^2$. We finish by controlling the waiting time from the first super-useful index (observed after the required waiting time) to the first super-useful index that maps to a trace-useful index. Let $\Delta_r$ denote the gap between the $r^{\text{th}}$ super-useful index and the $(r+1)^{\text{st}}$ super-useful index, and $Y_r$ denote the indicator of the event that the $r^{\text{th}}$ super-useful index did not map to a trace-useful index. Then, by Lemma \ref{useful&trace-useful}, Lemma \ref{notuseful&trace-useful}, and near independence, we have $\mathbb{E}[\Delta_r Y_r] \le \exp(-cT)$, so Markov implies $\mathbb{P}[\sum_r \Delta_r Y_r > \exp(-(c/2) T)n] \le \exp(-(c/2) Tn)$. Thus, combining all three waiting times, we see that we have at most $\frac{1}{3}\frac{1}{K_2^{0.1}}n$ missing bits with high probability, thereby completing the proof. 

\vs

\section{An Application to Coded Trace Reconstruction}

In \cite{bls}, it was proven that, for any $\epsilon > 0$, there is a set $S \sub \{0,1\}^n$ of $2^{(1-\ep)n}$ strings so that a constant number $T_\ep$ of traces suffice to reconstruct an arbitrary string chosen from $S$ (the player is assured that the unknown string is in $S$). They showed that one can take $T_\ep = \exp(C\sqrt[3]{\log (1/\ep)})$. In this section, we recover their result, with our set $S$ being a pseudo-random set of strings, any two of which have edit distance at least $\epsilon n$. We do not however make any guarantees that $S$ is efficiently encodable or decodable unlike in their paper.

\vs

\begin{theorem}
For any $\delta,q \in (0,1)$, there is some $C = C(\delta,q)$ so that for any $\ep > 0$ and $n \ge 1$, there is a subset $S \sub \{0,1\}^n$ of size $|S| \ge 2^{(1-\ep)n}$ so that any unknown $s \in S$ can be reconstructed from $\exp(C\sqrt[3]{\log(1/\ep)})$ independent traces. 
\end{theorem}

\begin{proof}
Fix $\ep > 0$ and $n \ge 1$. By choosing $C$ large enough, we may assume $\ep < 10^{-6}$ and $n \ge 10^6$, say.

\vs

Perform the following (random) algorithm. Choose a string $s_1 \in \{0,1\}^n$ uniformly at random, and remove from $\{0,1\}^n$ all strings that are within $\ep^2 n$ edit distance of $s_1$. For $k \ge 1$, with $s_1,\dots,s_k$ chosen, choose $s_{k+1}$ uniformly at random from the remaining strings (which will be those not within $\ep^2 n$ edit distance of any $s_j$) and remove all remaining strings that are within $\ep^2 n$ edit distance of $s_{k+1}$; stop the algorithm if $s_{k+1}$ cannot be chosen (due to there being no remaining strings). As a crude bound, since $\ep$ is small and $n$ is large, this algorithm generates $s_1,\dots,s_m$ with $m \ge 2\cdot 2^{(1-\ep)n}$. By ignoring some of the $s_j$'s, we may assume that always $m := 2\cdot 2^{(1-\ep)n}$ strings $s_i$ are outputted.

\vs

Now, let $A$ be the function guaranteed by Theorem \ref{main}; that is, $A$ takes in $T := \exp(C\sqrt[3]{\log(1/\ep)})$ traces and outputs a string that is within $\ep^2 n/2$ edit distance of the source string, with probability at least $0.95$. 

\vs

Let $\mc{D}$ be the distribution over tuples of strings $(s_1,\dots,s_m)$ produced by the algorithm. Note that choosing $(s_1,\dots,s_m)$ according to $\mc{D}$ and then choosing $s_j$ uniformly at random from $\{s_1,\dots,s_m\}$ is equivalent to choosing a string uniformly at random from $\{0,1\}^n$. Therefore, $$\subtext{s_1,\dots,s_m \sim \mc{D}}{\mathbb{P}}\hs\hs\hs \subtext{s_j \sim \{s_1,\dots,s_m\}}{\mathbb{P}}\hs\hs\hs\subtext{U^{(1)},\dots,U^{(T)} \sim s_j}{\mathbb{P}} \left[d_e\left(A(U^{(1)},\dots,U^{(T)}),s_j\right) < \ep^2 n/2\right] \ge 0.95.$$ So there is a choice of $(s_1,\dots,s_m)$ in the support of $\mc{D}$ so that $$\subtext{s_j \sim \{s_1,\dots,s_m\}}{\mathbb{P}}\hs\hs\hs\subtext{U^{(1)},\dots,U^{(T)} \sim s_j}{\mathbb{P}} \left[d_e\left(A(U^{(1)},\dots,U^{(T)}),s_j\right) < \ep^2 n/2\right] \ge 0.95.$$ Therefore, there is a subset $S \sub \{s_1,\dots,s_m\}$ of size $|S| \ge m/2$ so that for any $s_j \in S$, it holds that \begin{equation}\label{bagel} \subtext{U^{(1)},\dots,U^{(T)} \sim s_j}{\mathbb{P}} \left[d_e\left(A(U^{(1)},\dots,U^{(T)}),s_j\right) < \ep^2 n/2\right] \ge 0.90.\end{equation} Note $|S| \ge 2^{(1-\ep)n}$. We define a function $B : (\{0,1\}^{\le n})^T \to S$ as follows: for a given $\vec{U} := (U^{(1)},\dots,U^{(T)})$, let $\widehat{x} = A(\vec{u})$ and have $B(\vec{u})$ be the $s \in S$ that is closest to $\widehat{x}$ in edit distance (with ties broken arbitrarily). Since the edit distance between any two strings in $S$ is more than $\ep^2 n$, we're done by \eqref{bagel}.
\end{proof}

\vs

\section{Acknowledgments}

We would like to thank Nina Holden for helpful conversations. Z.C. would like to thank Ray Li for mentioning the connection to coded trace reconstruction.

\Addresses


\begin{thebibliography}{10}

\bibitem{alonspencer}
N. Alon and J. H. Spencer. The Probabilistic Method,
Third Edition. Wiley-Interscience series in discrete
mathematics and optimization. Wiley, 2008.

\bibitem{bcfrs} 
F. Ban, X. Chen, A. Freilich, R. Servedio, and S. Sinha. Beyond trace reconstruction: population recovery from the deletion channel. ArXiv e-prints, April 2019, 1904.05532.

\bibitem{bcss}
Frank Ban, Xi Chen, Rocco A. Servedio, and Sandip Sinha. Efficient average-case
population recovery in the presence of insertions and deletions. In APPROX/RANDOM 2019, volume 145 of LIPIcs, pages 44:1–44:18. Schloss Dagstuhl-Leibniz-Zentrum für Informatik, 2019.

\bibitem{bfrps}
V. Bhardwaj, M. Franceschetti, R. Rao, P. A. Pevzner, and Y. Safonova. Automated analysis of immunosequencing datasets reveals novel immunoglobulin D genes across diverse species. PLoS Computational Biology, vol. 16, no. 4, p. e1007837, 2020.

\bibitem{bprs} V. Bhardwaj, P. A. Pevzner, C. Rashtchian, and Y. Safonova. Trace Reconstruction Problems in
Computational Biology. IEEE Transactions on Information Theory, pages 1–1, 2020.

\bibitem{br}
T. Brailovskaya, M. R\'{a}cz. Tree trace reconstruction using subtraces. ArXiv e-prints, February 2021, 2102.01541. 

\bibitem{bkkm}
T. Batu, S. Kannan, S. Khanna, and A. McGregor. Reconstructing strings from random traces. \textit{In Proceedings of the Fifteenth Annual ACM-SIAM Symposium on Discrete Algorithms, pages 910–918. ACM, New York}, 2004.

\bibitem{blikstad}
J. Blikstad. On the longest common subsequence of Thue-Morse words. Information Processing Letters, Volume 164, 2020.  

\bibitem{bls}
J. Brakensiek, R. Li, and B. Spang. Coded trace reconstruction in a constant number of traces.CoRR, abs/1908.03996, 2019.

\bibitem{chasetrlower}
Z. Chase. New lower bounds for trace reconstruction. To appear in \textit{Annales Institute Henri Poincare: Probability and Statistics}, May 2019, 1905.03031.

\bibitem{chasetrupper}
Z. Chase. New upper bounds for trace reconstruction. ArXiv e-prints, September 2020, 2009.03296.

\bibitem{cdlss}
X. Chen, A. De, C. Lee, R. Servedio, S. Sinha. Polynomial-time trace reconstruction in the smoothed complexity model. ArXiv e-prints, August 2020, 2008.12386.

\bibitem{cdrv}
M. Cheraghchi, J. Downs, J. Ribeiro, A. Veliche. Mean-based trace reconstruction over practically any replication-insertion channel. ArXiv e-prints, February 2021, 2102.09490.

\bibitem{cgor} 
M. Cheraghchi, R. Gabrys, O. Milenkovic, J. Ribeiro. Coded Trace Reconstruction. In \textit{IEEE Transactions on Information Theory}, doi: 10.1109/TIT.2020.2996377.

\bibitem{drr} 
S. Davies, M. R\'acz, and C. Rashtchian.  
 Reconstructing trees from traces. In \textit{Proceedings of the 32nd Conference On Learning Theory}, pp. 961-978. PMLR, 2019.
 

\bibitem{drrs} 
S. Davies, M. R\'acz, C. Rashtchian, B. Schiffer. Approximate trace reconstruction. ArXiv e-prints, December 2020, 2012.06713.

\bibitem{dds}
A. De, R. O’Donnell, and R. A. Servedio. Optimal mean-based algorithms for trace reconstruction. In \textit{STOC’17—Proceedings of the 49th Annual ACM SIGACT Symposium on Theory of Computing}, pages 1047–1056. ACM, New York, 2017.  

\bibitem{gsz}
E. Grigorescu, M. Sudan, M. Zhu. Limitations of mean-based algorithms for trace reconstruction at
small distance. ArXiv e-prints, November 2020, 2011.13737.

\bibitem{holdenlyons}
N. Holden and R. Lyons. Lower bounds for trace reconstruction. To appear in \textit{Annals of Applied Probability}, 2019.

\bibitem{hppz}
N. Holden, R. Pemantle, Y. Peres, A. Zhai. Subpolynomial trace reconstruction for random strings and arbitrary deletion probability. \textit{Mathematical Statistics and Learning} 2, no. 3 (2020): 275-309. Conference version in \textit{Proceedings of the 31st Conference On Learning Theory}, PMLR 75:1799-1840, 2018.

\bibitem{hmpw}
T. Holenstein, M. Mitzenmacher, R. Panigrahy, and U. Wieder. Trace reconstruction with constant deletion probability and related results. In \textit{Proceedings of the Nineteenth Annual ACM-SIAM Symposium on Discrete Algorithms, pages 389–398. ACM, New York}, 2008.

\bibitem{katoh}
K. Katoh, G. Asimenos, H. Toh. Multiple alignment of DNA sequences with MAFFT. Methods Mol Biol. 2009; 537:39–64. PubMed: 19378139.

\bibitem{kmmp} 
A. Krishnamurthy, A. Mazumdar, A. McGregor, S. Pal. Trace reconstruction: generalized and parameterized. ArXiv e-prints, April 2019, 1904.09618.

\bibitem{lcaplus}
R, Lopez, Y. Chen, S. Ang, S. Yekhanin, K. Makarychev, M. Racz, G. Seelig, K. Strauss, L. Ceze. DNA assembly for nanopore data storage readout. Nature Communications, 10(1):1–9, 2019.

\bibitem{mr}
T. Maranzatto, L. Reyzin. Reconstructing arbitrary trees from traces in the tree edit
distance model. ArXiv e-prints, February 2021, 2102.03173.

\bibitem{mpv} 
A. McGregor, E. Price, and S. Vorotnikova. Trace reconstruction revisited. In \textit{Proceedings of the 22nd Annual European Symposium on Algorithms}, pages 689–700, 2014.

\bibitem{mitzenmacher}
M. Mitzenmacher. A survey of results for deletion channels and related synchronization channels. Probability Surveys, 6:1–33, 2009.

\bibitem{narayanan}
S. Narayanan. Population recovery from the deletion channel: Nearly matching trace reconstruction bounds. CoRR, abs/2004.06828, 2020.

\bibitem{circulartr}
S. Narayanan, M. Ren. Circular Trace Reconstruction. ArXiv e-prints, September 2020, 2009.01346.

\bibitem{nazarovperes}
F. Nazarov and Y. Peres. Trace reconstruction with $\exp(O(n^{1/3}))$ samples. In \textit{STOC’17— Proceedings of the 49th Annual ACM SIGACT Symposium on Theory of Computing}, pages 1042–1046. ACM, New York, 2017. 

\bibitem{pereszhai}
Y. Peres and A. Zhai. Average-case reconstruction for the deletion channel: subpolynomially many traces suffice. In 58th Annual IEEE Symposium on Foundations of Computer Science—FOCS
2017, pages 228–239. IEEE Computer Soc., Los Alamitos, CA, 2017.
MR3734232

\bibitem{simabruck}
J. Sima, J. Bruck. Trace reconstruction with bounded edit distance. ArXiv e-prints, February 2021, 2102.05372. 

\end{thebibliography}
\end{document}